\documentclass[11pt]{article}

\usepackage{latexsym,enumerate}
\usepackage{amsmath,amsthm,amssymb}

\newtheorem{thm}{Theorem}
\newtheorem{lemma}[thm]{Lemma}
\newtheorem{prop}[thm]{Proposition}
\newtheorem*{claim}{Claim}

\newcommand{\Bin}{\mathrm{Bin}}
\newcommand{\eps}{\varepsilon}
\newcommand{\cC}{\mathcal{C}}
\newcommand{\cP}{\mathcal{P}}
\newcommand{\cS}{\mathcal{S}}
\newcommand{\cSp}{\mathcal{S}'}
\newcommand{\Tm}{T_{\mathrm{mix}}}
\newcommand{\diam}{\mathrm{diam}}
\newcommand{\dTV}{d_{\mathrm{TV}}}

\newcommand{\pG}{G^*}

\newcommand{\Ds}{D^*}
\newcommand{\ds}{\delta^*}
\newcommand{\Ks}{K^*}

\renewcommand{\le}{\leqslant}
\renewcommand{\ge}{\geqslant}

\begin{document}

\title{Smoothed analysis on connected graphs.\footnote{Preliminary version of this work was presented at RANDOM, 2014, Barcelona.}}

\author{
  Michael Krivelevich\thanks{School of Mathematical Sciences, Tel Aviv University, Tel Aviv 69978, Israel. Research supported in part by the USA-Israel BSF Grant 2010115 and by grant 912/12 from the Israel Science Foundation. E-mail address: {\tt krivelev@post.tau.ac.il}.}
  \and
  Daniel Reichman\thanks{Department of Computer Science, Cornell University, Ithaca, NY. Supported in part by The Israel Science Foundation (grant No.\ 621/12). E-mail address: {\tt daniel.reichman@gmail.com}.}
  \and
  Wojciech Samotij \thanks{School of Mathematical Sciences, Tel Aviv University, Tel Aviv 69978, Israel; and Trinity College, Cambridge CB2 1TQ, UK. Research supported in part by ERC Advanced Grant DMMCA and grants from the Israel Science Foundation. E-mail address: {\tt samotij@post.tau.ac.il}.}
}

\maketitle

\thispagestyle{empty}

\begin{abstract}
 The main paradigm of smoothed analysis on graphs suggests that for any large graph $G$ in a certain class of graphs, perturbing slightly the edge set of $G$ at random (usually adding few random edges to $G$) typically results in a graph having much ``nicer'' properties. In this work we study smoothed analysis on trees or, equivalently, on connected graphs. Given an $n$-vertex connected graph $G$, form a random supergraph $\pG$ of $G$ by turning every pair of vertices of $G$ into an edge with probability $\frac{\eps}{n}$, where $\eps$ is a small positive constant. This perturbation model has been studied previously in several contexts, including smoothed analysis, small world networks, and combinatorics.

  Connected graphs can be bad expanders, can have very large diameter, and possibly contain no long paths. In contrast, we show that if $G$ is an $n$-vertex connected graph, then typically $\pG$ has edge expansion $\Omega(\frac{1}{\log n})$, diameter $O(\log n)$, vertex expansion $\Omega(\frac{1}{\log n})$, and contains a path of length $\Omega(n)$, where for the last two properties we additionally assume that $G$ has bounded maximum degree. Moreover, we show that if $G$ has bounded degeneracy, then typically the mixing time of the lazy random walk on $\pG$ is $O(\log^2 n)$. All these results are asymptotically tight.
\end{abstract}


\setcounter{page}{1}

\section{Introduction}

In this paper, we consider the following model of randomly generated graphs. We are given a fixed undirected graph $G=(V,E)$ on $n$ vertices. For every pair $f \in {V \choose 2}$, we add $f$ to $G$, independently of all other pairs, with probability $\frac{\eps}{n}$, where $\eps$ is a small (yet fixed) positive constant. Let $R$ be the set of edges added and consider the random graph
\[
\pG := (V,E \cup R).
\]
This model can be viewed as a generalization of the classical Erd\H{o}s--R\'{e}nyi random graph, where one starts from an empty graph and adds edges between all possible pairs of vertices independently with a given probability. The focus on ``small'' $\eps$ means that we are interested in the effect of a rather gentle random perturbation. In particular, the average degree of $\pG$ is (typically) close to that of $G$ (assuming that $G$ is connected, for example). Studying the effect of small perturbations on graphs, matrices, and other structures arises in diverse settings in several fields such as combinatorics, design and analysis of algorithms, linear algebra, and mathematical programming. We refer the reader to Section~\ref{sec:related-works} for more details.

 In this work, we study several properties of $\pG$, when $G$ is \emph{connected}. We first need a few definitions.  For a graph $G=(V,E)$ and a subset $S \subseteq V$, we denote by $\partial S$ the set of all edges of $G$ with exactly one endpoint in $S$. We define $N(S)$ to be the set of all vertices in $V \setminus S$ that have a neighbor in $S$. When the graph $G$ is not clear from the context, we will use the notation $\partial_G S$ and $N_G(S)$ to avoid ambiguity. The \emph{edge}-isoperimetric number of $G$ (also known as the \emph{Cheeger constant}), denoted $c(G)$, is defined by
\[
c(G):=\min \left\{ \frac{|\partial(U)|}{|U|} \colon 0<|U|\le \frac{|V|}{2} \right\}.
\]
Similarly, the \emph{vertex}-isoperimetric number of $G$, denoted $\iota(G)$, is defined by
\[
\iota(G):=\min \left\{ \frac{|N(U)|}{|U|} \colon 0<|U|\le \frac{|V|}{2} \right\}.
\]
Somewhat informally we shall refer to $c(G)$ and $\iota(G)$ as the edge and the vertex expansions of $G$, respectively. Observe that $c(G) / \Delta(G) \le \iota(G) \le c(G)$, where $\Delta(G)$ is the maximum degree of $G$. Hence, when $\Delta(G)$ is bounded by a constant, then the vertex and edge expansions of $G$ have the same order of magnitude. On the other hand, there are graphs $G$ with $\Delta(G)$ not bounded by a constant (such as stars) for which $\iota(G) = O(c(G)/\Delta(G))$. The Cheeger constant has been studied extensively as it is related to a host of combinatorial properties of the underlying graph. In particular, there is a strong connection between the Cheeger constant of $G$ and the mixing time of the lazy random walk on $G$.

\subsection{Our results}

We begin by describing our results regarding the expansion properties of perturbed connected graphs.

For every \emph{connected} $n$-vertex graph $G$, it holds that $\iota(G) = \Omega(\frac{1}{n})$ as every subset $S \subseteq V$ has at least one neighbor outside $S$. Moreover, if $G$ is a tree, then $\iota(G) = O(\frac{1}{n})$. Our first result is that for every \emph{connected} graph with bounded maximum degree, the random perturbation $\pG$ asymptotically almost surely\footnote{That is, with probability tending to $1$ as the number of vertices $n$ tends to infinity.} (a.a.s.) satisfies $\iota(\pG) = \Omega(\frac{1}{\log n})$.

\begin{thm}
  \label{thm:tree}
  There exists a constant $\delta > 0$ such that the following holds. Let $G$ be an $n$-vertex connected graph with maximum degree $\Delta$. If $R \sim G(n,\frac{\eps}{n})$ for some $\eps = \eps(n) \le 1$, then a.a.s.\ the graph $\pG=G \cup R$ has vertex expansion at least $\frac{\delta \eps}{\Delta^3\log n}$.
\end{thm}

We note that in general one cannot remove restrictions on the maximum degree entirely. To see this, consider the case when $G = K_{1,n-1}$. After adding to $G$ \emph{any} $\eps n$ edges, there will be an independent set $S$ with at least $(1-2\eps)n$ vertices such that $|N(S)| = 1$.

We obtain a similar bound on the \emph{edge}-expansion without any assumptions on the maximum degree.

\begin{thm}
  \label{thm:expansion}
  For every $\eps > 0$ and $\alpha < 1$, there exists $\delta > 0$ such that the following holds. Let $G$ be an $n$-vertex connected graph, choose $R \sim G(n,\frac{\eps}{n})$, and let $\pG = G \cup R$. Then a.a.s.\ for every set $S \subseteq V(G)$ with $|S| \le \alpha n$,
  \[
  |\partial_{\pG} S| \ge \frac{\delta}{\log (en/|S|)} |S|.
  \]
  In particular, $c(\pG) \ge \frac{\delta}{\log (en)}$.
\end{thm}

It should be noted that Theorem~\ref{thm:expansion} implies that a.a.s.\ the vertex expansion of $\pG$ is at least $\frac{\delta}{\Delta(\pG)\log(en)}$. This improves the bound obtained in Theorem~\ref{thm:tree} when $\Delta(G) \gg \left(\frac{\log n}{\log\log n}\right)^{1/3}$; to see this, observe that a.a.s.\ the maximum degree of $G(n,\eps/n)$ is $\Theta\left(\frac{\log n}{\log \log n}\right)$.


Furthermore, we prove an even stronger bound on the edge expansion of connected subsets of a perturbed connected graph.

\begin{thm}
  \label{thm:expansion-connected}
  For every $\eps > 0$ and $\alpha < 1$, there exist $\delta > 0$ and $K > 0$ such that the following holds. Let $G$ be an $n$-vertex connected graph, choose $R \sim G(n, \eps/n)$, and let $\pG = G \cup R$. Then a.a.s.\ for every connected (in $\pG$) set $S \subseteq V(G)$ with $K \log n \le |S| \le \alpha n$,
  \[
  |\partial_{\pG} S| \ge \delta |S|.
  \]
\end{thm}
We consider sets of size at most $\alpha n$ for an arbitrary $\alpha<1$ (instead of restricting our attention to sets of size at most $n/2$, as is customary in dealing with edge expansion) since this allows us later to give upper bounds on the conductance of sets of volume up to a half of the total volume which is crucial for the proof of Theorem~\ref{thm:main} stated below. Here volume is measured in terms of the degree sum rather than the number of vertices.

Using Theorem~\ref{thm:expansion-connected}, we derive the following upper bound on the diameter of a randomly perturbed connected graph. Observe that the diameter of a (non-perturbed) $n$-vertex connected graph may be as high as $n-1$ (when the graph is a path on $n$ vertices).

\begin{thm}
  \label{thm:diameter}
  For every $\eps > 0$, there exists $C > 0$ such that the following holds. Let $G$ be an $n$-vertex connected graph, choose $R \sim G(n,\frac{\eps}{n})$, and let $\pG = G \cup R$. Then a.a.s.\ the diameter of $\pG$ is at most $C \log n$.
\end{thm}

Flaxman and Frieze~\cite{FlaxmanFrieze} proved an upper bound of $O(\log n)$ on the diameter of randomly perturbed strongly connected \emph{digraphs}. The result of~\cite{FlaxmanFrieze} requires that the maximum degree of the base graph is upper bounded by some function of $n$ (see Section~\ref{sec:related-works} for details). Unlike their work, our upper bound on the diameter of $\pG$ holds unconditionally, regardless of the maximum degree of the base graph $G$.

Using Theorem~\ref{thm:expansion-connected}, we also prove upper bounds on the mixing times of lazy random walks on randomly perturbed connected graphs. Recall the notion of \emph{degeneracy}. Given a positive integer $D$, a graph $G$ is called \emph{$D$-degenerate} if every subgraph of $G$ contains a vertex of degree at most $D$. Observe that every graph $G$ is $\Delta(G)$-degenerate and trees are $1$-degenerate. Also, if $G$ is $D$-degenerate, then every subset $S \subseteq V(G)$ spans at most $D|S|$ edges. Using the machinery developed by Fountoulakis and Reed~\cite{FoRe07}, we are able to prove the following bound on the mixing time of the lazy random walk on a random perturbation of a connected graph with bounded degeneracy.

\begin{thm}
  \label{thm:main}
  For all positive $D$ and $\eps$, there exists a constant $M$ such that the following holds. Let $G$ be an $n$-vertex $D$-degenerate connected graph, choose $R \sim G(n, \frac{\eps}{n})$ and let $\pG = G \cup R$. Then a.a.s.
  \[
  \Tm(\pG) \le M \log^2 n.
  \]
\end{thm}

For a precise definition of $\Tm$, we refer the reader to Section~\ref{sec:preliminaries}. The bound in Theorem~\ref{thm:main} above is tight when $G$ is the path on $n$ vertices, as then a.a.s.\ $\pG$ contains an induced subgraph which is a path of length $\Omega(\log n)$. Moreover, we cannot expect that $\Tm(\pG) = O(\log^2 n)$ for an arbitrary connected graph $G$, as the following example demonstrates. Let $G$ be the graph obtained by connecting two disjoint cliques of order $n/2$ with a single edge and let $R \sim G(n,\frac{1}{n})$. As the number of edges interconnecting the two cliques in the perturbed graph is a.a.s.\ $O(n)$, the conductance of $\pG$ is $O(\frac{1}{n})$, which implies via standard results (e.g., \cite{Peres}) that the mixing time of the lazy random walk on $\pG$ is $\Omega(n)$.

The effect of small random perturbations on connected graphs from several families has been studied before, see, e.g., \cite{Addario,NewmanWatts1}. In particular, a $O(\log^2n)$ bound (holding a.a.s.) on the mixing time of a simple random walk on a random perturbation of the ring graph was proved in~\cite{Addario}, see Section~\ref{sec:related-works} for more details. Our Theorem~\ref{thm:main} demonstrates that an upper bound of $O(\log^2n)$ on the mixing time (holding a.a.s.) is a rather general phenomenon for perturbed connected graphs.

Finally, we establish the existence of long paths in perturbed connected graphs with bounded maximum degree. Observe that a connected bounded degree graph with $n$ vertices might contain only paths of length $O(\log n)$, as the case of the complete binary tree demonstrates.
\begin{thm}
  \label{thm:long}
  For every $\eps, \Delta > 0$, there exists $c>0$ such that the following holds. Let $G$ be an $n$-vertex connected graph with maximum degree bounded by $\Delta$. Form a random graph $R \sim G(n,\frac{\eps}{n})$ and let $\pG = G \cup R$. Then $\pG$ a.a.s.\ contains a path of length $cn$.
\end{thm}
The assumption that the maximum degree is bounded is crucial, as it is easy to see that if $G = K_{1,n-1}$ and $\eps < 1$, then a.a.s.\ the length of a longest path in $G \cup R$ is $O(\log n)$. This follows as it is known that a.a.s.\ each connected component of $G(n,\frac{\eps}{n})$ has $O(\log n)$ vertices and the vertex set of any simple path in $\pG$ intersects at most two connected components in $R$.

Finally, one may ask what happens if one incorporates edge deletions in our model. Consider the case when $G$ is an $n$-vertex tree with $\Omega(n)$ leaves, e.g., $G$ is a complete binary tree over $n$ vertices. If we now add and remove edges randomly with probability $\frac{\eps}{n}$, then with constant probability, we will isolate one of the leaves of $G$ (as the probability of isolating a fixed vertex with degree one in $G$ is about $\frac{\eps}{n} \cdot e^{-\eps})$. Hence we cannot expect the resulting after perturbed graph to have nontrivial expansion properties.

\subsection{Our techniques}

In proving Theorems~\ref{thm:tree} and~\ref{thm:long}, we use a fairly basic result  (see e.g., \cite{KN}) to decompose graph of bounded degree to disjoint connected sets of comparable sizes. Treating each of these sets as a `super-vertex' allows us to view the auxiliary graph induced by the random edges between sets as essentially the standard binomial random graph whose edge probability should be now compared to the number of super-vertices as opposed to the (much larger) number of vertices. Consequently, standard methods and results regarding the threshold for connectivity and the existence of long paths in binomial random graphs can be used.

In order to deal with the Cheeger constant of perturbed graphs, we derive an upper bound on the number of connected subsets of given cardinality and number of vertices in their boundary. Recall that a subset of vertices of a graph is \emph{connected}, if it induces a connected subgraph.

\begin{prop}
  \label{prop:Cvab}
  Let $G$ be an arbitrary graph and let $v \in V(G)$. For integers $a$ and $b$, let $\cC(v, a, b)$ denote the collection of connected subsets $A$ of $V(G)$ such that $v \in A$, $|A| = a$, and $|N(A)| = b$. Then
  \[
  |\cC(v, a, b)| \le \binom{a+b-1}{b}.
  \]
\end{prop}

The bound in Proposition~\ref{prop:Cvab} is tight for all values of $a$ and $b$. To see this, consider the case when $G = K_{1, a+b-1}$ and $v$ is the center vertex.

In bounding the mixing time, we rely on an upper bound on the mixing time of a lazy random walk due to Fountoulakis and Reed~\cite{FoRe07}. This bound, which they used~\cite{FoRe08} to upper bound the mixing time of the lazy random walk on the giant component of $G(n,p)$, is suited for bounding the mixing time of random walks on graphs whose large vertex sets expand well but small sets (e.g., of logarithmic size) do not have to. Another attractive feature of the result of Fountoulakis and Reed is that it allows one to focus on the conductance of connected sets, which significantly simplifies union bound estimates. We note that the classical work of Jerrum and Sinclair \cite{JS} for upper-bounding the mixing time $\Tm$ in terms of the conductance $\Phi$ of $G$ (see Section~\ref{sec:preliminaries} for precise definitions), namely
\[
\Tm \le O\left(\frac{\log n}{\Phi^2}\right)
\]
would give in our setting a weaker bound of $O(\log ^3 n)$.

\subsection{Related work}
\label{sec:related-works}

The study of random perturbations of graphs arose in several contexts. One of them is the field of \emph{smoothed analysis}, which originated from the work of Spielman and Teng~\cite{SpTe} on the smoothed complexity of the simplex algorithm. This field attempts to provide a theoretical explanation for the good performance of certain heuristics on ``real-life'' instances based on the assumption that they are likely to be subjected to random perturbations. It has been applied to a host of other problems such as numerical analysis and linear algebra~\cite{Sankar,Tao}, machine learning~\cite{blumd}, and satisfiability~\cite{COFeKrVi, feige}. It is closely related to the study of random perturbations of combinatorial structures and devising efficient algorithms for such ``semi-random'' instances, which had been considered in the past, see~\cite{Blum, BF, FeigeKilian, Flaxman, KST, Spencer, SuVo}.

Another context where the study of random perturbations naturally arose, is the field of small world networks, see~\cite{Durrett, NewmanWatts1, NewmanWatts2}. In an attempt to model social networks arising in ``real-life'' settings, one studies properties of networks composed of a (usually sparse) connected ``base'' graph along with a set of random edges, where every random edge is added independently with probability $p$. One well-known example is the Newmann--Watts small world model~\cite{NewmanWatts1, NewmanWatts2} (NW small world for short), where the base graph is the $(n,k)$-ring, i.e., the graph with vertex set $\{0, \ldots, n-1\}$ and edge set $\big\{\{i,j\}: i+1\le j \le i+k \big\}$ (where addition is modulo $n$) and $p$ is equal to $\frac{c}{n}$ for some constant $c>0$.

Durrett~\cite{Durrett} showed that with high probability the mixing time of the lazy random walk on the NW small world is upper-bounded by $O(\log^3 n)$ and lower-bounded by $\Omega(\log^2 n)$. These results were improved by Addario-Berry and Lei~\cite{Addario} who proved that this mixing time is a.a.s.\ $O(\log^2 n)$. It is worth noting that our approach is similar to~\cite{Addario} in the sense that we bound the conductance of connected sets and then use this upper bound with the results of~\cite{FoRe07} to bound the mixing time. The crucial difference between our proof and theirs is the technique of counting connected sets with small boundary. While~\cite{Addario} uses a somewhat involved argument based on the Lagrange inversion formula, we use a more elementary approach based on Proposition~\ref{prop:Cvab}.

Similar ideas were used in the study of the mixing time of the simple random walk on the giant component in a supercritical random graph $G(n, \frac{1+\eps}{n})$. Fountolakis and Reed~\cite{FoRe07} and Benjamini, Kozma, and Wormald~\cite{BKW} showed that a.a.s.\ this mixing time is $O(\log^2 n)$. Moreover, there has been interest in probability theory in studying the robustness of the mixing time under random perturbations, see~\cite{BM,Ding}.

Flaxman~\cite{Flaxman} examined the edge expansion of several models of randomly perturbed graphs. In particular, he considered the model studied in this work. He showed in particular that if $G=(V,E)$ is an $n$-vertex connected graph and $R \sim G(n,\frac{\eps}{n})$, then a.a.s.\ all linear sized vertex subsets $S \subseteq V$, $|S|\le n/2$, send outside at least a linear in $n$ number of edges in $G^*=G\cup R$.
The effect of adding random edges on the diameter of a given graph was considered by Bollob\'{a}s and Chung~\cite{BC}, who proved that adding a random matching to an $n$-vertex cycle result a.a.s.\ with a graph with diameter $(1+o(1))\log_2 n$. The case of \emph{directed} graphs was considered by Flaxman and Frieze~\cite{FlaxmanFrieze}. They proved that if $D$ is an $n$-vertex strongly connected digraph with maximum degree bounded by $n^{\frac{\eps}{100}}$ and $R \sim D(n, \frac{\eps}{n})$, then a.a.s.\ the diameter of $D \cup R$ is at most $100 \eps^{-1} \log n$. Our proof idea is different from theirs.



\subsection{Outline of the paper}

In Section~\ref{sec:preliminaries}, we fix some notation, give a precise definition of the mixing time of a random walk, and state two auxiliary probabilistic lemmas that are used later in the paper. In Sections~\ref{sec:vertex} and~\ref{sec:edge-expansion} we prove Theorems~\ref{thm:tree} and~\ref{thm:expansion}, respectively. Section~\ref{sec:edge-expansion} contains also the a proof of Proposition~\ref{prop:Cvab}. In Section~\ref{sec:exp_conn} we prove Theorem~\ref{thm:expansion-connected}. Building upon Theorem~\ref{thm:expansion-connected}, we prove Theorem~\ref{thm:diameter} (in Section \ref{sec:diameter}) and Theorem~\ref{thm:main} (in Section~\ref{sec:mixing-time}). The proof of Theorem~\ref{thm:long} can be found in Section~\ref{sec:long}. In Section~\ref{sec:concluding-remarks}, we state several concluding remarks.

\section{Preliminaries}

\label{sec:preliminaries}

Let $G$ be a graph with vertex set $V$. Given two disjoint sets $A, B \subseteq V$, we denote by $E(A, B)$ the set of all edges with one endpoint in $A$ and one endpoint in $B$ and by $E(A)$ the set of all edges entirely contained in $A$. We will denote the cardinality of $E(A)$ by $e(A)$. The degree of a vertex $v$ in $G$ is denoted by $\deg(v)$ and the maximum degree of $G$ is denoted by $\Delta(G)$.


We denote by $[n]$ the set $\{1, \ldots, n\}$. When dealing with an $n$-vertex graph, we will implicitly assume that that its vertex set is $[n]$. We denote by $G(n,p)$ the classical binomial random graph with vertex set $[n]$ and edge probability $p$. Given a graph property $\cP$ and a sequence $(\mu_n)$, where $\mu_n$ is a probability distribution over $n$-vertex graphs, we will say that $\cP$ holds asymptotically almost surely (a.a.s.) if $\lim_{n \to \infty} \Pr_{G \sim \mu_n}(G \in \cP)=1.$

The lazy random walk on a graph $G = (V,E)$ is the Markov chain defined as follows. The set of states is $V$. For any vertex $u \in V$, the walk stays in $u$ with probability $\frac{1}{2}$ and with probability $\frac{1}{2}$, it moves to a uniformly chosen random neighbor $v$ of $u$ (so that the transition probability $\Pr(u \to v)$ is $\frac{1}{2\deg(u)}$). When $G$ is connected, this Markov chain is well-known to be irreducible and ergodic and hence it converges to a stationary distribution $\pi$ which can be seen to satisfy $\pi(u)=\frac{\deg(u)}{2|E|}$ for every $u \in V$, see~\cite{Peres}. We will be interested in estimating how quickly this random walk on $G$ converges to its stationary distribution $\pi$. To this end, we recall that the \emph{total variation distance} $\dTV$ between two distributions $p_1$, $p_2$ on $V$ is defined by
\[
\dTV(p_1, p_2) := \max_{A \subseteq V} |p_1(A) - p_2(A)|.
\]
Let $P$ be the transition matrix of the random walk. The mixing time $\Tm(G)$ is defined by
\[
\Tm(G):= \sup_{x_0}\min \left\{ t \colon d_{TV}(x_0P^t,\pi) \le \frac{1}{4} \right\},
\]
where the supremum is taken over all probability distributions $x_0$ on $V$.

The following lemma is standard. We provide a proof for the sake of completeness.
\begin{lemma}
  \label{lemma:eS}
  For every $C \ge 1$, if $p \le C/n$, then a.a.s.\ for every non-empty set $S$ of vertices in $G(n,p)$, we have $e(S) < 2C|S|$. In particular, if $p \le \frac{1}{n}$, then a.a.s.\ for every non-empty set $S$ of vertices in $G(n,p)$, we have $e(S) < 2|S|$.
\end{lemma}

\begin{proof}
  For a fixed set $S$ of size $k$,
 \[
 \Pr\big(e(S) \ge 2Ck\big) \le \binom{\binom{k}{2}}{2Ck} p^{2Ck} \le \left(\frac{ekp}{4C}\right)^{2Ck}
  \]
  and hence, letting $\mathcal{E}$ denote the event that $e(S) \ge 2C|S|$ for some $S \neq \emptyset$,
  \[
  \begin{split}
    \Pr(\mathcal{E}) & \le \Pr\big(e(G(n,p)) \ge Cn\big) + \sum_{k = 5}^{n/2} \binom{n}{k}\left(\frac{ekp}{4C}\right)^{2Ck} \\
   & \le \sum_{k=5}^{n/2} \left(\frac{en}{k} \cdot \left(\frac{ekp}{4C}\right)^{2C}\right)^{k} + o(1) = o(1).
  \end{split}
  \]
 To see the last equality, note that if $k \le \sqrt{n}$, then $np^{2C}k^{2C-1} = O(1/\sqrt{n})$ and if $k \le n/2$, then
  \[
  \frac{en}{k} \cdot \left(\frac{ekp}{4C}\right)^{2C} \le 2e \cdot \left(\frac{e}{8}\right)^{2C} \le \frac{e^3}{32} \le \frac{2}{3}.\qedhere
  \]
\end{proof}

We close this section with a version of Chernoff's inequality (see, e.g., \cite{MU}).

\begin{lemma}
  \label{lemma:Chernoff}
  Suppose that $X = \sum_{i=1}^mX_i$, where every $X_i$ is a $\{0,1\}$-valued random variable with $\Pr(X_i=1)=p$ and the $X_i$s are jointly independent. Then for each $\eta \in (0,1)$,
  \[
  \Pr\big( X<(1-\eta)pm \big) \le \exp(-pm\eta^2/2).
  \]
\end{lemma}

\section{Vertex expansion}
\label{sec:vertex}

In this section, we prove Theorem~\ref{thm:tree}.

\begin{proof}[{Proof of Theorem~\ref{thm:tree}}]
  Let $G$ be a connected graph with $n$ vertices and maximum degree $\Delta$ and suppose that $\eps \le 1$. We may assume that $\Delta^3 / \eps \ll n / \log n$ as otherwise the assertion of the theorem follows immediately from the fact that $G$ is connected and hence $\iota(G^*) \ge \iota(G) = \Omega(\frac{1}{n})$. Let
  \[
  k = \frac{C\Delta\log n}{\eps},
  \]
  where $C$ is an absolute constant, which we will define later. Partition the vertex set of $G$ into \emph{disjoint} pieces $V_1, \ldots, V_t$, such that for each $1 \le i \le t$, we have $k \le |V_i| \le \Delta k$ and $G[V_i]$ is \emph{connected}. This is fairly straightforward, see, e.g., \cite[Proposition~4.5]{KN}. Call each $V_i$ a~\emph{blob}. Observe that
  \[
  \frac{n}{\Delta k} \le t \le \frac{n}{k}.
  \]

  Suppose that $R \sim G(n,\frac{\eps}{n})$ and let $G^* = G \cup R$. The probabilistic statement about the random graph $R$ that we need is the following one.

  \begin{claim}
    A.a.s.\ for every non-empty $I \subseteq [t]$ with $|I| \le t/2$, there are at least $|I|/2$ blobs with indices outside of $I$ that are connected by an edge of $R$ to the set $\bigcup_{i\in I} V_i$.
  \end{claim}
  \begin{proof}
    Let
    \[
    \rho = \frac{\eps k^2}{2n}.
    \]
    Clearly, the probability that two blobs $V_i$ and $V_j$ are connected in $R$ is $1-(1-\frac{\eps}{n})^{|V_i|\cdot|V_j|}$, which is at least $\rho$, provided that $n$ is sufficiently large, as $\rho = o(1)$ by our assumption that $\Delta^3 \le \eps n \le n$. The probability $P$ that there exists a non-empty set $I \subseteq [t]$ with $|I| \le t/2$ such that the set $\bigcup_{i\in I} V_i$ has an edge (in $R$) to fewer than $|I|/2$ blobs outside of $I$ satisfies
    \[
    P \le \sum_{1 \le j \le t/2}\binom{t}{j}\binom{t-j}{\lfloor j/2 \rfloor} (1-\rho)^{(t-\lfloor \frac{3j}{2} \rfloor)j} \le \sum_{1 \le j \le t/2} t^{\frac{3j}{2}+1}\cdot \exp\left(-\frac{\rho j t}{4} \right).
    \]
    It is easy to verify that $P = o(1)$ if $C > 20$.
  \end{proof}

  We may assume that $R$ has the property from the statement of the claim. We claim that $G^*$ has vertex expansion at least $\frac{\delta \eps}{\Delta^3 \log n}$ for some absolute positive constant $\delta$. Fix a set $A \subseteq [n]$ with $|A|\le n/2$ and denote
  \begin{align*}
    I_0 & = I_0(A)=\{ i \in [t] \colon V_i\subseteq A\}, \\
    I_1 & = I_1(A)=\{ i \in [t] \colon \emptyset \ne V_i\cap A\ne V_i\}, \\
    I_2 & = I_2(A)=\{i \in [t] \setminus I_0 \colon \text{$V_i$ has a neighbor in $A$}\}.
  \end{align*}
  In other words, $I_0$ is the set of (indices of) blobs fully contained in $A$, $I_1$ is the set of blobs having at least one vertex in $A$ but not falling completely inside $A$, and finally $I_2$ is the set of blobs outside $I_0$ having a neighbor in $A$. We now develop inequalities relating the sizes of $I_0,I_1,I_2$. These inequalities allow us to directly estimate $|N(A)|$.

  It follows from the assumed property of $R$ that a.a.s.
  \begin{equation}
    \label{eq:I2-lb}
    |I _2|\ge \min\left\{\frac{|I_0|}{2}, \frac{t-|I_0|}{3}\right\}.
  \end{equation}
  This is clear when $|I_0| \le t/2$; we simply take $I = I_0$. Else, we let $I = [t] \setminus (I_0 \cup I_2)$, note that $|I| \le t - |I_0| \le t/2$, and observe that no blob in $I$ is connected to a blob in $I_0$ (and hence the neighborhood of $I$ must be completely contained in $I_2$). Observe crucially that $|N(A)|\ge |I_1 \cup I_2|$. Indeed, $A$ has at least one neighbor in each set $V_i \setminus A$ with $i \in I_1 \cup I_2$. This follows from connectivity of $V_i$ when $i \in I_1$ and from the fact that $V_i \cap A = \emptyset$ when $i \in I_2 \setminus I_1$. As $|A| \le \frac{n}{2}$ and $A^c \subseteq \bigcup_{i \not\in I_0} V_i$, we have $|I_0| \le (1-\frac{1}{2\Delta})t$. In particular, it follows from~\eqref{eq:I2-lb} that $|I_2| \ge \frac{|I_0|}{6\Delta}$. Finally, note that $A \subseteq \bigcup_{i\in I_0\cup I_1} V_i$, implying that
  \[
  |I_0| + |I_1| \ge \frac{|A|}{\Delta k}.
  \]
  Now, if $|I_1| \ge |I_0|$, then $|N(A)| \ge |I_1| \ge \frac{|A|}{2\Delta k}$. Otherwise, $|N(A)| \ge |I_2| \ge \frac{|I_0|}{6\Delta} \ge \frac{|A|}{12\Delta^2 k}$. Hence, $\pG$ has vertex expansion at least $\frac{\delta \eps}{\Delta^3 \log n}$, where $\delta$ is an absolute positive constant.
\end{proof}

\noindent{\bf Remark.} Observe that the exact same proof as above works if instead of assuming that the graph $G$ is connected and has maximum degree bounded by $\Delta$, we assume that $\Delta(G)\le \Delta$ and all connected components of $G$ are at least as large as $C \Delta \log n / \eps$, where $C$ is a large enough constant.

\section{Edge expansion}

\label{sec:edge-expansion}

In this section, we prove Proposition~\ref{prop:Cvab} and derive from it Theorem~\ref{thm:expansion}.
\begin{proof}[Proof of Proposition~\ref{prop:Cvab}]
  Assume that the vertices of $G$ are labeled with distinct integers. We will describe an algorithm that, given an $A \in \cC(v,a,b)$, outputs an encoding of $A$ using a sequence of $a-1$ ones and $b$ zeros in such a way that no two sets are encoded with the same sequence. This will clearly imply the statement of the proposition.

  Let $S = \{v\}$ and $B = \emptyset$. The algorithm will grow the sets $S$ and $B$, adding one vertex to one of the sets in each of its $a+b-1$ iterations, making sure that the invariants $S \subseteq A$ and $B \subseteq N(A)$ hold in every iteration. It will stop when $S = A$ and $B = N(A)$, after having moved $a-1$ vertices to $S$ and $b$ vertices to $B$. For the sake of brevity, we will denote by $T$ the set $N(S) \setminus B$, updated after each iteration. Intuitively, in every iteration, $S$ is the set of vertices that are known to belong to $A$, $B$ is the set of vertices that are known to belong to $N(A)$, and $T$ is the remaining set of vertices for which we do not yet know whether they belong to $A$ or to $N(A)$.

  While $S \neq A$ or $B \neq N(A)$, we repeat the following. Let $w$ be the vertex with the smallest label in $T$. Note that the assumption that $A$ is connected implies that $T$ is non-empty. Consider two cases. If $w \in A$, move $w$ to $S$ and append $1$ to the sequence encoding $A$. Otherwise, if $w \not\in A$, then move $w$ to $B$ and append $0$ to the sequence encoding $A$. Note that in this case $w \in N(A)$, since $S \subseteq A$ and $w \in N(S) \setminus A$.

  A moment of thought reveals that decoding can be performed in an analogous way and given $v$ and the $\{0,1\}$-sequence encoding $A$, one can recover the set $A$. This completes the proof.
\end{proof}

After this work was completed, Noga Alon pointed out to us that one may derive Proposition~\ref{prop:Cvab} from the following classical inequality proved by Bollob\'as~\cite{Bo65}.

\begin{thm}[{\cite{Bo65}}]
  \label{thm:set-pairs}
  Suppose that $A_1, \ldots, A_m$ and $B_1, \ldots, B_m$ be sets such that $|A_i| = k$ and $|B_i| = \ell$ for all $i \in \{1, \ldots, m\}$. If furthermore
  \begin{itemize}
  \item
    $A_i \cap B_i = \emptyset$ for all $i \in \{1, \ldots, m\}$,
  \item
    $A_i \cap B_j \neq \emptyset$ for all distinct $i,j \in \{1, \ldots, m\}$,
  \end{itemize}
  then $m \le \binom{k+\ell}{\ell}$.
\end{thm}

Indeed, in order to derive the claimed upper bound on $|\cC(v,a,b)|$, one may invoke Theorem~\ref{thm:set-pairs} with $k = a-1$, $\ell = b$, $m = |\cC(v,a,b)|$, and $\{(A_i, B_i)\}_{i=1}^m$ being the set of all pairs $(A \setminus \{v\}, N(A))$ with $A \in \cC(v,a,b)$. It is straightforward to check that the assumptions of the theorem are satisfied.

\begin{proof}[Proof of Theorem \ref{thm:expansion}]
  For positive integers $s$, $m$, and $b$, denote by $\cS(s, m, b)$ the collection of all sets $S$ of $s$ vertices such that in the graph $G$, the set $S$ induces exactly $m$ connected components and the sum of their vertex boundaries is exactly $b$. In other words, $\cS(s,m,b)$ consists of all sets $S \subseteq V(G)$ such that there is a partition $S = S_1 \cup \ldots \cup S_m$, where each $S_i$ is connected, there are no edges of $G$ connecting different $S_i$, and $|N_G(S_1)| + \ldots + |N_G(S_m)| = b$. Since $G$ is connected, each $S \in \cS(s, m, b)$ satisfies $|\partial_G S| \ge b \ge m \ge 1$ (but not necessarily $|N_G(S)| \ge b$). Therefore, it is enough to show that there exist positive constants $K$ and $\delta$ such that a.a.s.\ for every $s$ satisfying $s \ge K \log n$,
  \begin{equation}
    \label{eq:partial-R-S}
    |\partial_R S| \ge \frac{\delta s}{\log (en/s)} \quad \text{for all $S \in \cS(s,m,b)$ with $m \le b \le \frac{\delta s}{\log (en/s)}$}.
  \end{equation}
(For small sets $S$ with $|S|\le K\log n$, we have that $|\partial_{\pG}S|\ge |\partial_G S|\ge 1\ge \frac{\delta s}{\log (en/s)}$, since we may assume that $K \delta\le 1/2$).
  In order to facilitate a union bound argument, we will estimate the size of $\cS(s, m, b)$ with small $b$ and $m$ using Proposition~\ref{prop:Cvab}. To this end, we first argue that each set in $\cS(s, m, b)$ can be exactly described by the following:
  \begin{enumerate}[(i)]
  \item
    a set $W = \{v_1, \ldots, v_m\}$ of $m$ vertices of $G$,
  \item
    a partition $s = s_1 + \ldots + s_m$, where $s_i \ge 1$ for each $i$,
  \item
    a partition $b = b_1 + \ldots + b_m$, where $b_i \ge 1$ for each $i$, and
  \item
    a set $S_i$ in $\cC(v_i, s_i, b_i)$ for each $i \in [m]$.
  \end{enumerate}

  To see this, note that we may assume that there is a canonical linear ordering on the vertices of $G$. The representation of $S$ in $\cS(s,m,b)$ as (i)--(iv) is natural. Indeed, given such an $S$, we find the unique partition $\{S_1, \ldots, S_m\}$ into connected components of $G[S]$ and arbitrarily choose one vertex from each $S_i$ to form $W$. We order the sets $S_1, \ldots, S_m$ according to the canonical linear ordering on their representatives $v_1, \ldots, v_m$. Finally, we let $s_i = |S_i|$ and $b_i = |\partial_G S_i|$. Observe that this mapping is not only injective, but actually each set $S$ can be represented in $s_1 \cdot \ldots \cdot s_m$ different ways.

  It follows from Proposition~\ref{prop:Cvab}, as well as from the inequality $\binom{x}{y}\binom{w}{z} \le \binom{x+w}{y+z}$, that
  \[
  \begin{split}
    |\cS(s, m, b)| & \le \binom{n}{m} \sum_{(s_i), (b_i)} \prod_{i=1}^m \binom{s_i + b_i - 1}{b_i} \le \binom{n}{m} \sum_{(s_i),(b_i)} \binom{\sum_i(s_i+b_i-1)}{\sum_i b_i} \\
    & \le \binom{n}{m} \binom{s-1}{m-1} \binom{b-1}{m-1} \binom{s+b-m}{b} \le \binom{n}{m}\binom{s}{m}\binom{b}{m}\binom{s+b}{b}. \\
  \end{split}
  \]
  Consequently, if $m \le b \le \delta s / \log (en/s) \le s$, then it follows from the well-known estimate $\binom{x}{y} \le \left(\frac{ex}{y}\right)^y$ and the fact that the function $y \mapsto (ex/y)^y$ is increasing on the interval $(0,x]$ that
  \[
  \begin{split}
    |\cS(s, m, b)| & \le \left(\frac{en}{m}\right)^m \left(\frac{es}{m}\right)^m \left(\frac{eb}{m}\right)^m \left(\frac{e(s+b)}{b}\right)^b \le \left(\frac{e^4 n s (s+b)}{b^3}\right)^b \\
    & \le \left(\frac{2 e^4 n (\log (en/s))^3}{\delta^3 s }\right)^{\frac{\delta s}{\log (en/s)}} \le \exp\left(C\delta \log(1/\delta) s\right),
  \end{split}
  \]
  where $C$ is some absolute constant.

  On the other hand, by Chernoff's inequality, for a fixed set $S$ with $|S| = s \le \alpha n$,
  \[
  \begin{split}
    \Pr(|\partial_R S| < \delta s) & \le \Pr\big(\Bin(s(n-s), \eps/n) < \delta s\big) \\
    & \le \Pr\big(\Bin((1-\alpha)sn, \eps/n) < \delta s\big) \le \exp\big(- (1-\alpha)\eps s/8 \big).
  \end{split}
  \]
  provided that $\delta < (1-\alpha)\eps/2$.

  Finally, choose positive constants $K$ and $\delta$ such that
  \[
  K \ge \frac{64}{\eps(1-\alpha)}, \qquad C\delta \log (1/\delta) \le \frac{\eps(1-\alpha)}{16}, \qquad \text{and} \qquad K\delta \le 1/2.
  \]
  Taking a union bound over all triples $b$, $m$, and $s$ satisfying $K \log n \le s \le \alpha n$ and $m \le b \le \delta s / \log (en/s)$, we get that
  \[
  \Pr(\text{property \eqref{eq:partial-R-S} fails}) \le n^3\exp\left[\left(C\delta \log(1/\delta) - \frac{\eps(1-\alpha)}{8}\right)s\right] = o(1).\qedhere
  \]
\end{proof}

\section{Edge expansion of connected sets}
\label{sec:exp_conn}
Here we prove Theorem~\ref{thm:expansion-connected}.
\begin{proof}
  Due to the obvious monotonicity we can assume that  $\eps <1$. Recall the definition of $\cS(s, m, b)$ from the proof of Theorem~\ref{thm:expansion}. It clearly suffices to show that a.a.s.\ for every $s$ with $K \log n \le s \le \alpha n$,
  \begin{equation}
    \label{eq:partial-R-S-conn}
    |\partial_R S| \ge \delta s \quad \text{for all \emph{connected} $S \in \cS(s,m,b)$ with $m \le b < \delta s$},
  \end{equation}
  where connected means connected in the graph $\pG$.

  Let us denote by $\cSp(s,m,b)$ the collection of all ordered pairs
  \begin{itemize}
  \item
    $S = S_1 \cup \ldots \cup S_m \in \cS(s,m,b)$, where $S_1, \ldots, S_m$ are connected components of $G[S]$,
  \item
    $m - 1$ pairs $\{s_1,t_1\}, \ldots, \{s_{m-1},t_{m-1}\}$ of vertices of $S$ whose addition to $G$ makes $G[S]$ connected.
  \end{itemize}
  A moment of thought reveals that for fixed $s$, the probability that~\eqref{eq:partial-R-S-conn} does not hold is bounded by
  \begin{equation}
    \label{eq:pr-nonexp-conn}
    \sum_{m=1}^{\delta s} \sum_{b = m}^{\delta s} |\cSp(s, m, b)| \cdot (\eps/n)^{m-1} \cdot \Pr\big(\Bin(s(n-s),\eps/n) \le \delta s \big).
  \end{equation}
  Therefore, it suffices to prove the following.
  \begin{claim}
    There exists an absolute constant $C$ such that for all $s$, $m$, and $b$ with $m \le b \le \delta s$,
    \[
    |\cSp(s,m,b)| \le n^m \exp\big( C \delta \log(1/\delta) s \big).
    \]
  \end{claim}
  Indeed, if $K \log n \le s \le \alpha n$, then by Chernoff's inequality,
  \[
  \begin{split}
    \Pr\big(\Bin(s(n-s), \eps/n) < \delta s\big) & \le \Pr\big(\Bin((1-\alpha)sn, \eps/n) < \delta s\big) \\
    \le \exp\big(- (1-\alpha)\eps s/8\big),
  \end{split}
  \]
  provided that $\delta < (1-\alpha)\eps/2$. Hence, \eqref{eq:pr-nonexp-conn} is bounded from above by
  \[
  s^2 n \exp\left[\left(C\delta \log(1/\delta) - \frac{(1-\alpha)\eps}{8}\right) s\right].
  \]
  If we choose $K$ and $\delta$ as in the proof of Theorem~\ref{thm:expansion}, a union bound over $K \log n < s \le \alpha n$ yields that~\eqref{eq:pr-nonexp-conn} is indeed $o(1)$.

  Hence, it suffices to prove the claim. To this end, we will argue that each element of $\cSp(s,m,b)$ can be uniquely described by the following:
  \begin{enumerate}[(i)]
  \item
    a set $W = \{v_1, \ldots, v_m\}$ of vertices of $G$,
  \item
    a partition $s = s_1 + \ldots + s_m$, where $s_i \ge 1$ for each $i$,
  \item
    a partition $b = b_1 + \ldots + b_m$, where $b_i \ge 1$ for each $i$,
  \item
    a set $S_i$ in $\cC(v_i, s_i, b_i)$ for each $i \in [m]$,
  \item
    a partition $m - 1 = d_1 + \ldots + d_m$, where $d_i \ge 0$ for each $i$,
  \item
    a multiset $D_i$ of $d_i$ elements from $S_i$ for each $i \in [m]$,
  \item
    a permutation $f \colon [m-1] \to [m-1]$.
  \end{enumerate}

  Assuming that this is indeed the case, by Proposition~\ref{prop:Cvab} we have
  \[
  \begin{split}
    |\cSp(s, m, b)| & \le \binom{n}{m} (m-1)! \sum_{(s_i), (b_i), (d_i)} \prod_{i=1}^m \left[ \binom{s_i+b_i-1}{b_i} \binom{s_i+d_i-1}{d_i} \right] \\
    & \le \frac{n^m}{m} \binom{s-1}{m-1}^2 \binom{b-1}{m-1} \binom{2m-2}{m-1} \binom{s+b-m}{b} \le (2n)^m \binom{s}{m}^2 \binom{b}{m} \binom{s+b}{b}.
  \end{split}
  \]
  Consequently, if $m \le b \le \delta s$, then
  \[
  |\cSp(s, m, b)| \le n^m \left(\frac{2e^4 s^2 (s+b)}{b^3}\right)^b \le n^m \left(\frac{3e^4}{\delta^3}\right)^{\delta s} \le n^m \exp\left(C\delta \log(1/\delta) s\right),
  \]
  where $C$ is some absolute constant.

  Finally, we show that each $S \in \cSp(s,m,b)$ may be uniquely described by (i)--(vii). First, observe that (i)--(iv) uniquely describe the set $S = S_1 \cup \ldots \cup S_m$, together with a root vertex $v_i$ in each connected component $S_i$, whose use will be explained later. As in the proof of Theorem~\ref{thm:expansion}, one may assume some canonical linear ordering $\preceq$ on the set of vertices of $G$. Given this ordering, one may canonically order the sets $S_1, \ldots, S_m$ according to the canonical ordering $\preceq$ on the set $\{\min_\preceq S_1, \ldots, \min_\preceq S_m\}$ of representatives of each $S_i$. Now, note that the $m - 1$ pairs of vertices of $S$ whose addition to $G$ makes $G[S]$ connected naturally define a tree $T$ on the vertex set $\{S_1, \ldots, S_m\}$. Root this tree at $S_m$ and orient all of its edges away from the root. Now, start with $v_m \in S_m$ and for each $i \in [m-1]$ let $v_i \in S_i$ be the unique vertex of $S_i$ that lies in the pair of vertices of $S$ that corresponds to the unique edge of $T$ going into $S_i$. Next, for each $i \in [m]$, let $d_i$ be the outdegree of $S_i$ in $T$ and let $D_i$ be the multiset of $d_i$ vertices of $S_i$ that lie in the pairs of vertices of $S$ that correspond to the $d_i$ edges of $T$ going out of $S_i$. Finally, let $D = D_1 \cup \ldots \cup D_m$ and observe that the $m-1$ pairs of vertices of $S$ that correspond to the edges of $T$ define a bijection between $D$ and $\{v_1, \ldots, v_{m-1}\}$. Namely, if $D = \{w_1, \ldots, w_{m-1}\}$, where $w_1 \preceq \ldots \preceq w_{m-1}$, then this bijection can be described by a permutation $f \colon [m-1] \to [m-1]$ defined by letting $f(i)$ be the unique $j$ such that $\{v_i, w_j\}$ is one of the $m-1$ pairs of vertices whose addition to $G$ makes $G[S]$ connected. This concludes the proof of the theorem.
\end{proof}

\section{Diameter}

\label{sec:diameter}

In this section, we prove Theorem~\ref{thm:diameter}. Since adding edges to a graph can only decrease its diameter, it suffices to consider the case when $G$ is a tree and $\eps \le 1/3$. Since $e(G) = n-1$, it follows from Chernoff's inequality (Lemma~\ref{lemma:Chernoff}) that a.a.s.\ $\pG$ has at most $(1+\eps)n$ edges. Hence, it is enough to prove that there is a constant $C = C(\eps)$ such that a.a.s.
\begin{equation}
  \label{eq:diameter-claim}
  e\big(B(v, C\log n)\big) > \frac{(1+\eps)n}{2} \qquad \text{for every $v \in V(G)$,}
\end{equation}
where $B(v,r)$ denotes the $\pG$-ball of radius $r$ around $v$. Indeed, \eqref{eq:diameter-claim} implies that for every $u, v \in V(G)$, we have that $B(u, C\log n) \cap B(v, C\log n) \neq \emptyset$, and consequently $\diam(\pG) \le 2 C \log n$.

Fix some $v \in V(G)$, let $K$ and $\delta$ be as in Theorem~\ref{thm:expansion-connected} with $\alpha = 3/4$, and condition on the event that $\pG$ satisfies the assertion of this theorem. Moreover, condition on the event that $R$ satisfies the assertion of Lemma~\ref{lemma:eS} with $C = 1$. This implies that $e(S)/3 \le |S| \le e(S) + 1$ for every connected set $S$ in $\pG$. Since $\pG$ is connected, we clearly have that $|B(v,r)| \ge r+1$. Hence, if $r \ge K \log n$, we have that
\[
e(B(v,r+1)) \ge \min \left\{\frac{3n}{4} - 1, \left(1 + \frac{\delta}{3}\right) e\big(B(v,r)\big) \right\}.
\]
Letting $C = K + \frac{1}{\log(1+\delta/3)}$, we have that
\[
e(B(v, C\log n)) \ge \frac{3n}{4} - 1 > \frac{2n}{3} \ge \frac{(1+\eps)n}{2},
\]
as claimed.
\hfill\qed
\vspace{-0.5em}
\section{Mixing time}

\label{sec:mixing-time}

In this section, we prove Theorem \ref{thm:main}. Let $\eps$ be a positive real, let $D$ be a positive integer, and assume that $G$ is a connected $D$-degenerate graph. Let $\pG = G \cup R$, where $R \sim G(n, \eps/n)$.

Our argument for bounding the mixing time is based on the approach of Fountoulakis and Reed~\cite{FoRe07, FoRe08}. The main idea there is that one can bound the mixing time of an abstract irreducible, reversible, and aperiodic Markov chain in terms of the \emph{conductances} of \emph{connected} sets of states of various sizes. For simplicity, we only state their results in the setting of the lazy random walk on the graph $\pG$. Let $\pi$ be the stationary distribution of this walk. For $S \subseteq V$, let $\pi(S)$ equal $\sum_{v \in S} \pi(v)$. It can be verified that $\pi(S)=\frac{2e_{\pG}(S) + |\partial_{\pG} S|}{2e(\pG)}$.
We define
\[
Q(S) = \sum_{u \in S, v \not \in S} \pi(u) \Pr(u \to v) = \frac{|\partial_{\pG} S|}{4e(\pG)}
\]
and note that $Q(S) = Q(S^c)$. The \emph{conductance} $\Phi(S)$ of $S$ is
\[
\Phi(S) = \frac{Q(S)}{\pi(S) \pi(S^c)} = \frac{|\partial_{\pG} S|}{2 \cdot (2e_{\pG}(S) + |\partial_{\pG} S|) \cdot \pi(S^c)}.
\]
Let $\pi_{\min} = \min_{v \in V(G)} \pi(v)$. For $p > \pi_{\min}$, we denote by $\Phi(p)$ the minimum conductance of a connected (in $\pG$) set $S$ with $p/2 \le \pi(S) \le p$ (if there is no such $S$, we define $\Phi(p) = 1$). Fountoulakis and Reed~\cite{FoRe07} proved the following result.

\begin{thm}
  \label{thm:mixing-time}
  There exists an absolute constant $C$ such that
  \[
  \Tm(\pG) \le C \sum_{j=1}^{\lceil \log_2 \pi_{\min}^{-1} \rceil} \Phi^{-2}(2^{-j}).
  \]
\end{thm}

In the remainder of the proof, we will estimate the sum in Theorem~\ref{thm:mixing-time}. We claim that it is enough to prove the following.

\begin{lemma}
  \label{lemma:main}
  There exist positive constants $\ds$ and $\Ks$ such that a.a.s.\ for every connected (in $\pG$) set $S$ with $\frac{\Ks \log n}{n} \le \pi(S) \le 1/2$,
  \[
  \Phi(S) \ge \ds.
  \]
\end{lemma}

Indeed, suppose that the assertion of Lemma~\ref{lemma:main} holds for some $\ds$ and $\Ks$. Let $J$ be the set of indices $j$ satisfying $2^{-j} \le \frac{2\Ks \log n}{n}$ and note the $|J^c| < \log_2 n$, as $2^{-j} > \frac{2\Ks \log n}{n}$ implies that $j < \log_2 n$. Since $\pG$ is connected, we have that for every set $S$,
\[
\Phi(S) \ge \frac{|\partial_{\pG} S|}{4e(\pG) \cdot \pi(S)} \ge \frac{1}{4e(\pG) \cdot \pi(S)}.
\]
Condition on the event that $R$ satisfies the assertion of Lemma~\ref{lemma:eS} with $C = \max\{\eps, 1\}$. Let $\Ds = D + 2C$ and observe that the degeneracy assumption implies that
\begin{equation}
  \label{eq:Ds-sparse}
  e_{\pG}(S) \le \Ds|S| \quad \text{for every $S \subseteq V(G)$}.
\end{equation}
In particular, $e(\pG) \le \Ds n$ and hence, letting $M = 129 (\Ks)^2 (\Ds)^2$,
\[
\begin{split}
  \sum_{j=1}^{\lceil \log_2 \pi_{\min}^{-1} \rceil} \Phi^{-2}(2^{-j}) & \le |J^c| \cdot (\ds)^{-2} + \sum_{j \in J} 2^{-2j} (4e(\pG))^2 \\
  & \le O(\log n) + 2 \cdot \max_{j \in J} \{2^{-2j}\} \cdot 16(\Ds)^2n^2 \le M \log^2 n,
\end{split}
\]
provided that $n$ is sufficiently large, where we used the definition of $J$ and the inequality $\sum_{j \ge i} 2^{-2j} \le 2^{-2i+1}$.

Therefore, it suffices to prove Lemma~\ref{lemma:main}. We first show that any connected set $S$ with $\pi(S) \le 1/2$ has at most $n - \Omega(n)$ elements.

\begin{claim}
  Every connected (in $\pG$) set $S \subseteq V(G)$ with $\pi(S) \le 1/2$ satisfies
  \[
  |S| \le \frac{\Ds n+1}{\Ds+1}.
  \]
\end{claim}
\begin{proof}
  Since $\pi(S) \le 1/2$ implies that $\pi(S) \le \pi(S^c)$, we have
  \[
  \begin{split}
    2e_{\pG}(S) & = 2e(\pG)\pi(S) - |\partial_{\pG} S| \le 2e(\pG)\pi(S^c) - |\partial_{\pG} S| \\
    & = 2e_{\pG}(S^c) \le 2\Ds|S^c| = 2\Ds(n-|S|).
  \end{split}
  \]
  Since $S$ is connected in $\pG$, we obtain $e_{\pG}(S) \ge |S|-1$ and the claim follows.
\end{proof}

Let $\delta$ and $K$ be as in Theorem~\ref{thm:expansion-connected} with $\alpha = \frac{\Ds+1}{\Ds+2}$ and condition on the event that $\pG$ satisfies the assertion of this theorem. Let
\[
\Ks = (2\Ds+1)K \qquad \text{and let} \qquad \ds = \min\left\{ \frac{1}{2\Ds + 1}, \frac{\delta}{2\Ds + \delta} \right\}.
\]
It follows from~\eqref{eq:Ds-sparse} that for every connected set $S$ with $\pi(S) \le 1/2$, we have
\[
\Phi(S) \ge \frac{|\partial_{\pG} S|}{2e_{\pG}(S) + |\partial_{\pG} S|} \ge \frac{|\partial_{\pG} S|}{2\Ds|S| + |\partial_{\pG} S|}.
\]
Note that if $|\partial_{\pG} S| \ge |S|$, then $\Phi(S) \ge \ds$, so we may assume otherwise. In particular, if $\pi(S) \ge \Ks\log n /n$, then, as $e(\pG)\ge n-1$, we have
\[
\Ks\log n \le 2\pi(S)e(\pG) = 2e_{\pG}(S) + |\partial_{\pG} S| \le (2\Ds + 1)|S|
\]
and hence $|\partial_{\pG} S| \ge \delta |S|$. It follows that $\Phi(S) \ge \ds$. This concludes the proof of Lemma~\ref{lemma:main} and therefore the proof of Theorem~\ref{thm:main}.
\hfill\qed

\section{Long Paths}

\label{sec:long}

Here we show that after adding random edges, each with probability $\frac{\eps}{n}$, to a connected $n$-vertex graph with bounded maximum degree, we a.a.s.\ get a path whose length is linear in $n$.

\begin{proof}[Proof of Theorem~\ref{thm:long}]
  Let $k$ be a sufficiently large constant. Similarly as in the proof of Theorem~\ref{thm:tree}, let us partition the vertex set of the graph $G$ into connected pieces (blobs) $V_1, \ldots, V_t$ such that for each $1\le i\le t$, we have $k \le |V_i| \le \Delta k$. As in Theorem~\ref{thm:tree}, the probability that two blobs are connected (in $\pG$) is at least $\frac{\eps k^2}{2n}$. Hence, if $k$ is sufficiently large, then the auxiliary graph naturally induced by the blobs (obtained by treating each blob as a super-vertex and connecting two super-vertices if there is an edge of $\pG$ connecting the two blobs) contains the random graph $G(t, C/t)$, where $C \to \infty$ as $k \to \infty$. It is well known (\cite{AKS}, see also~\cite{KS}) that if $C > 1$, then $G(t, C/t)$ a.a.s.\ contains a path $P_0$ of length $\Omega(t)$. Since the blobs are connected and $t \ge n/(\Delta k)$, one can turn $P_0$ into a path $P$ in $\pG$, whose length is at least as large as the length of~$P_0$. Indeed, we may use the edges of $P_0$ to move between the blobs and the edges of $G$ to connect the entry and the exit points of $P_0$ within each blob traversed by $P_0$.
\end{proof}
\section{Concluding remarks}

\label{sec:concluding-remarks}

In this paper, we studied the model of randomly perturbed connected graphs. Using a new general upper bound on the number of connected subsets with small vertex boundary, we proved lower bounds on edge expansion under mild assumptions on the base graph. We established several other interesting properties of randomly perturbed connected graphs such as bounds on the diameter and the mixing time of the lazy random walk. It would be interesting to study other parameters of this model.

It seems that randomly perturbed connected $n$-vertex graph with bounded degeneracy shares some similarities with the giant component in the supercritical Erd{\H{o}}s--R{\'e}nyi random graph $G(n,\frac{1+\eps}{n})$. In particular, a.a.s.\ they both have diameter $O(\log n)$, mixing time $O(\log^2 n)$, and contain paths of length $\Omega(n)$. It could be interesting to explore this analogy further and to check whether the methods used in this work to study the model of randomly perturbed graphs can be applied to the other model.

\medskip

\noindent
\textbf{Acknowledgments.} We would like to thank Uri Feige, Jon Kleinberg, Gady Kozma, and Ofer Zeitouni for motivating discussions. We thank Yuval Peres for pointing out an inaccuracy in a previous version of this work and Noga Alon for pointing out the relation between our Proposition~\ref{prop:Cvab} and the result of Bollob\'as~\cite{Bo65}. Finally, we thank two anonymous referees for their careful reading of the paper and several insightful comments and suggestions.

\bibliographystyle{amsplain}
\bibliography{tree-exp}

\end{document}